\documentclass[12pt,centertags,oneside]{amsart}
\usepackage{amsmath,amstext,amsthm,amscd,typearea,hyperref}
\usepackage{amssymb}
\usepackage{a4wide}
\usepackage[mathscr]{eucal}
\usepackage{mathrsfs}
\usepackage{typearea}
\usepackage{charter}
\usepackage{pdfsync}
\usepackage[a4paper,width=16.2cm,top=3cm,bottom=3cm]{geometry}

\numberwithin{equation}{section}

\newtheorem{theorem}{Theorem}[section]

\newtheorem{proposition}[theorem]{Proposition}

\newtheorem{lemma}[theorem]{Lemma}
\newtheorem{remark}[theorem]{Remark}

\newcommand{\dist}{\mathop{\mathrm{dist}}\nolimits}

\newcommand{\ddc}{dd^c}
\newcommand{\dc}{d^c}

\def\mL{\mathcal{L}}

\newcommand{\dbar}{\overline\partial}

\newcommand{\C}{\mathbb{C}}
\newcommand{\D}{\mathbb{D}}

\newcommand{\R}{\mathbb{R}}

\renewcommand{\S}{\mathbb{S}}

\renewcommand{\H}{\mathbb{ H}}

\newcommand{\abs}[1]{\lvert#1\rvert}

\emergencystretch=\maxdimen
\title{Harmonic currents directed by foliations by Riemann surfaces}

\author{Tien-Cuong Dinh and Hao Wu}
\address{Department of Mathematics, National University 
of Singapore, 10 Lower Kent Ridge Road, Singapore 119076.}
\email{matdtc@nus.edu.sg; e0011551@u.nus.edu}

\thanks{This work is  supported by NUS Tier 1  
Grants R-146-000-248-114 and R-146-000-319-114 from National University of
Singapore}

\date{}
\begin{document}

\begin{abstract}
We study local positive $\ddc$-closed currents directed by a foliation by Riemann surfaces near a hyperbolic singularity which have no mass on the separatrices. 
A theorem of Nguy\^en says that  the Lelong number of such a current at the singular point vanishes. We prove that this property is sharp: one cannot have any better mass estimate for this current near the singularity. 
\end{abstract}

\maketitle

\medskip

\noindent {\bf Classification AMS 2010}: 37F75, 37A.

\medskip


\medskip


\section{Introduction}

In theory of foliations by Riemann surfaces, directed positive $\ddc$-closed currents play a central role like invariant measures for a dynamic system, see e.g.\ \cite{bernd-sibony,forn-sibony,forn-sibony-wold,garnett,nguyen2020singular,sullivan}. A fundamental problem is to understand such currents near the singularities of the foliations.

Let $\mathscr F$ be a foliation by Riemann surfaces near $0\in \C^n$ such that $0$ is an isolated hyperbolic singularity. Let $T$ be a positive $\ddc$-closed current of bi-dimension $(1,1)$ directed by $\mathscr F$. We assume that this current has no mass on the separatrices of the foliation at $0$. In \cite{Nguyen,nguyen2020singular},  Nguy\^en proves that the Lelong number of $T$ at $0$ vanishes.
Equivalently, the mass of $T$ in the polydisc $\delta \D^n:=\{(z_1,z_2,\cdots,z_n)\in \C^2:\,\abs{z_i}<\delta\}$ satisfies $$\|T\|_{\delta \D^n}=o(\delta^2)\quad\text{as}\quad \delta \to 0.$$
In this paper, we show that Nguy\^en's result is sharp: one cannot have a better estimate.

For simplicity, we consider the case of complex dimension $n=2$. It is not difficult to extend the result to the higher dimension case. Here is our main theorem.

\begin{theorem}\label{main}
Let $\mathscr F$ be a foliation by Riemann surfaces in a neighborhood of $0\in \C^2$. Assume that $0$ is a hyperbolic singularity.	 Let $\varepsilon:[0,1]\to\R^+$ be a continuous function such that $\varepsilon(0)=0$.
Then there exists a  positive $\ddc$-closed $(1,1)$-current $T$ in a neighborhood of $0\in \C^2$ directed by $\mathscr F$ having no mass on the separatrices at $0$ such that
$$\|T\|_{\delta\D^2}\geq \varepsilon(\delta) \delta^2.$$
\end{theorem}

In particular, we do not have the general estimate $$\|T\|_{\delta\D^2}\lesssim |\log\delta|^{-\alpha} \delta^2$$ for some $\alpha>0$ in the local setting. This estimate is crucial in the study of global dynamics of  foliations via Poincar\'e metric and random walk on leaves, see \cite{Nguyen2} for details.

Let $F$ be a holomorphic vector field on a neighborhood of $0\in \C^2$ which defines the foliation $\mathscr F$. Recall that $0$ is a hyperbolic singularity of $\mathscr F$ if we can choose $F$ so that $$F=\eta z_1 {\partial\over \partial z_1} +z_2{\partial\over \partial z_2}+\text{higher order terms}$$ with $\eta=a+ib$ and $b\neq 0$. By Poincar\'e-Dulac theorem, see e.g.\ \cite{arn}, holomorphic vector fields are linearizable near hyperbolic singularities. We can change the local coordinate system $(z_1,z_2)$ so that
$$F= \eta z_1{\partial \over\partial z_1}+z_2{\partial\over \partial z_2}.$$ From now on, we use these new coordinates. Then the leaves of $\mathscr F$ can be described explicitly, see Section \ref{pre} below. The two axes $\{z_1=0\}$ and $\{z_2=0\}$ are called the separatrices of $\mathscr F$ at $0$.

We will construct an explicit  positive $\ddc$-closed current $T$ having full mass on a leaf $\mL$ of $\mathscr F$. It is given by a positive harmonic function  on $\mL$.

In Section \ref{pre}, we will recall the standard parametrization of the leaves of $\mathscr F$ together with some properties that we need in this paper. We will also give the construction of the current $T$. In Section  \ref{proof}, we will show that $T$  satisfies our main theorem.

Throughout this paper, the symbols $\lesssim$ and $\gtrsim$ stand for inequalities up to a multiplicative
constant. We write $\simeq$ if both hold. We denote by $\H,\D,\delta\D,\delta\D^2$ the upper half-plane in $\C$, the unit disk, the disk of cent $0$ and radius $\delta$, and the bidisc $\delta\D\times\delta\D$ respectively.

\medskip
\section{Construction of harmonic current and some estimates}\label{pre}

We will use the notation as in \cite[Sections 4 and 5]{DS}. For simplicity, we also assume that $b>0$ (if $b<0$ we use the change $(z_1,z_2)\mapsto(z_2,z_1)$ and $F\mapsto \eta^{-1} F$ in order to reduce to the case $b>0$). Define the annulus $\mathbb A$ by $$\mathbb A:=\{\alpha\in\C:\,\,e^{-2\pi b}<\abs\alpha\leq 1\}$$ and  the sector $\S$ in the upper half-plane $\H$ by
$$\S:=\{\zeta=u+iv:\,\, v>0,bu+av>0\}.$$  For $\alpha\in \C^*$, consider the Riemann surface $\mL_\alpha$ immersed in $\C^2$ defined by $$z_1=\alpha e^{i\eta(\zeta+\log\abs\alpha/b)}\quad\text{and}\quad z_2=e^{i(\zeta+\log\abs\alpha/b)}\quad\text{with}\quad \zeta=u+iv.$$ The map $\zeta\mapsto (z_1,z_2)$ is injective because $\eta\notin \R$. It is easy to check that $\mL_\alpha$ is tangent to the vector field $F$. Hence $\mL_\alpha$ is a parametrization of a leaf of $\mathscr F$ in $\C^2$. 

For $\alpha_1,\alpha_2\in \mathbb A$, one can check that $\mL_{\alpha_1}$ and $\mL_{\alpha_2}$ are disjoint if $\alpha_1\neq \alpha_2$. The union of all $\mL_\alpha,\alpha\in \mathbb A$, is equal to $(\C^*)^2$. The intersection $L_\alpha:=\mL_\alpha \cap \D^2$ is given by the same equations as $\mL_\alpha$ but with $\zeta\in \S$. Since $L_\alpha$ is a connected submanifold of $(\D^*)^2$, it is a leaf of $\mathscr F\cap \D^2$.

Fix an $\alpha\in\mathbb A$,  and denote   by $\pi:\S\to L_\alpha$ the above parametrization of $L_\alpha$. From now on, we  take $\alpha=1$ for simplicity. In this case, we have 
\begin{equation}\label{2.1}
	\pi(\zeta):=(e^{i\eta\zeta}, e^{i\zeta})\quad\text{with}\quad \zeta=u+iv\in \S.
	\end{equation} 

 Consider the biholomorphic map $\Phi:\S \to \H$ defined by 
$$\Phi(\zeta):=\zeta^\gamma \quad \text{with}\quad \gamma>1 \quad\text{and}\quad \tan{\pi\over\gamma}=-{b\over a}.$$ We use the coordinate $Z=U+iV:=\Phi(\zeta)$ on $\H$.

We can replace the function $\varepsilon$ in the main theorem by a suitable larger function which is smooth on $(0,1]$, strictly increasing and concave. Fix a large constant $A>0$ and define the function $\widetilde H(x)$ on $\R$ by
$$\widetilde H(\pm t^\gamma):=\gamma^{-1}Ae^{-t}\varepsilon'(e^{-t}) \quad\text{for}\quad t\geq0.$$ We can easily check that
\begin{equation}\label{2.2}
\int_{x\leq -t^\gamma}\widetilde H(x) (-x)^{-1+1/\gamma} dx=\int_{x\geq t^\gamma}\widetilde H(x) x^{-1+1/\gamma} dx =A \varepsilon(e^{-t}).
\end{equation}
Then we extend $\widetilde H$ to a positive harmonic function, still denoted by $\widetilde H$, on the upper half-plane $\H$ by using Poisson kernel, i.e.
$$\widetilde H(U+iV):=\frac{1}{\pi}\int_\R \widetilde H(x) {V\over V^2+(U-x)^2} dx \quad\text{for}\quad U+iV\in \H.$$ Since $\widetilde H(x)=\widetilde H(-x)$, the function $\widetilde H$ is symmetric: $\widetilde H(U+iV)=\widetilde H(-U+iV)$ for $U+iV\in\H$.
 Define the function $H:=\widetilde H\circ \Phi$ on $\S$.

Consider the following $(1,1)$-current on $\D^2$,
$$T:=\pi_*(H[\S]).$$
In the rest of this paper, we will show that $T$ is a current that satisfies Theorem \ref{main}.

\vskip 5pt
As in \cite{DS},  to simplify the computation, we will introduce some new variables. Firstly,  let $\zeta^*:=u^*+i=-a/b+i$ be the intersection point of $\{v=1\}$ with the line $\{bu+av=0\}$. Define $$\rho:=|\zeta^*|^\gamma=-\Phi(\zeta^*).$$ We  consider $\zeta=u+iv$ in  the half-line $\{v=s\}\cap\S$  and the new variable $$r:=u/s-u^*.$$ Note that $rs$ is  comparable with the distance of $\zeta=u+iv$ to the edge $\{bu+av=0\}$ of $\S$ and we have $bu+av=bsr$ and $du=sdr$ on the half-line $\{v=s\}\cap\S$.   Define also the variables $$Z':=s^{-\gamma}Z,\quad U':=s^{-\gamma}U,\quad V':=s^{-\gamma}V$$ for $U+iV=\Phi(\zeta)\in\H$. Notice that $\zeta=s(\zeta^*+r)$ and hence $Z'=U'+iV'=\Phi(\zeta^*+r)$. Finally, we write $$x':=s^{-\gamma}x\quad\text{for}\quad x\in\R.$$ We need the following three lemmas for the new variables. See \cite[Section 5]{DS} for the proofs.

\begin{lemma}\label{UV1}
	When $r\to 0$, we have $U'=-\rho+O(r)$ and $V'=\beta r+O(r^2)$ for some constant $\beta>0$. Moreover, given a constant $N>0$, we have for $0\leq r\leq N$,
	$$ \dist(x',U'+iV')^2\geq c_N\big[ r^2+\dist(x',-\rho)^2\big],$$ where $\dist$ denotes the standard distance and $c_N>0$ is a constant independent of $x'$.
\end{lemma}

\begin{lemma}\label{UV2}
	 When $r\to \infty$, we have $U'=r^\gamma+O(r^{\gamma-1})$ and  $V'=\gamma r^{\gamma -1}+O(r^{\gamma-2}).$
\end{lemma}

\begin{lemma}\label{U'V'}
	There is a constant $c>0$ such that $$\int_{0}^{\infty}{V'\over V'^2+(U'-x')^2} dr\leq c |x'|^{-1+1/\gamma} \quad\text{for}\quad|x'|\geq 2\rho.$$
\end{lemma}

We also need the following  estimate.

\begin{lemma} \label{rgeq1}
There is a constant $c>0$ such that 
$$\int_{1/b}^{\infty}   {V'\over V'^2+(U'-x')^2} dr\geq c x'^{-1+1/\gamma} \quad\text{for}\quad x'\geq 1.$$
\end{lemma}

\begin{proof}
	The lemma is clear when $x'$ is bounded by a constant. So it is enough to consider $x'\geq (1/b)^\gamma$. In this case, we have $x'^{1/\gamma}\geq 1/b$ and the considered integral is larger than the integral for $r$ between $x'^{1/\gamma}$ and  $x'^{1/\gamma}+1$. For those $r$, by using  Lemma \ref{UV2}, we have $V'\simeq x'^{1-1/\gamma}$ and $r^\gamma=x'+O(x'^{1-1/\gamma})$ by mean value theorem. Hence   $|U'-x'|\lesssim V'$. Therefore, we have
	\[\int_{1/b}^{\infty}   {V'\over V'^2+(U'-x')^2} dr\gtrsim\int_{x'^{1/\gamma}}^{x'^{1/\gamma}+1} {V'\over V'^2}dr\simeq \int_{x'^{1/\gamma}}^{x'^{1/\gamma}+1} {1\over x'^{1-1/\gamma}} dr=x'^{-1+1/\gamma}.\]	
The proof of the lemma is finished.
\end{proof}

\medskip
\section{Proof of the main theorem}\label{proof}

Now we complete the proof of Theorem \ref{main}. Firstly, we prove that $T$ is a well-defined positive $\ddc$-closed $(1,1)$-current on $\D^2$.

\begin{proposition}\label{well-define}
$T$ is a positive $(1,1)$-current of finite mass in $\D^2$ supported by $\overline L_1=L_1\cup \{z_1z_2=0\}$.
\end{proposition}
\proof
The positivity of $T$ is clear. In order to show that $T$ is a current, it is enough to check for every smooth $(1,1)$-form $\phi$ with compact support in $\D^2$ that $\int_\S H(\zeta) \pi^*(\phi)$ is meaningful. For this purpose, we only need to show that 
$$\int_\S H(\zeta) \pi^*(\ddc\|z\|^2)<+\infty.$$ Indeed, this inequality also shows that $T$   has finite mass. It is clear that $T$ has full mass on $L_1$ and its support is  $\overline L_1.$

By a direct computation, using \eqref{2.1}, we have 
$$\pi^*(idz_1\wedge d\overline z_1)=(a^2+b^2)e^{-2(bu+av)}id\zeta\wedge d\overline\zeta$$ and 
$$\pi^*(idz_2\wedge d\overline z_2)= e^{-2v} id\zeta\wedge d\overline\zeta.$$ Recall that $\ddc=\frac i\pi \partial\dbar$. So we get
\begin{equation}\label{3.0}
\pi^*(\ddc\|z\|^2)=\frac 1 \pi\big((a^2+b^2)e^{-2(bu+av)}+e^{-2v} \big)id\zeta\wedge d\overline\zeta.
\end{equation}
 Note that on the half sector $\S_1:=\{bu+av\geq v\}\cap \S$, we have $e^{-2(bu+av)}\leq e^{-2v}$. Moreover,  the equation $bu+av=v$ is equivalent to $br=1$ in $\S$, which means that the intersection point of  $\{bu+av=s\}$ with $\{v=s\}$ corresponds to $r=1/b$ for any $s$. So using that $id\zeta\wedge d\overline\zeta=2du\wedge dv$ and the  variables $s,r$ introduced in Section \ref{pre}, we get
$$ \int_{\S_1} H(\zeta) \pi^*(\ddc\|z\|^2)\lesssim \int_{\S_1} H(\zeta)e^{-2v}id\zeta\wedge d\overline\zeta=2\int_0^\infty e^{-2s} \Big(\int_{1/b}^\infty s \widetilde H(U+i V) dr\Big) ds.$$
For the integral inside the parentheses, using Poisson formula and the variables $U',V',x'$ defined in Section \ref{pre}, we have
\begin{align*}
\int_{1/b}^\infty s \widetilde H(U+i V) dr&=\int_{1/b}^\infty \frac1\pi\int_\R s \widetilde H(x) {V\over V^2+(U-x)^2} dxdr\\
&=\frac1\pi\int_\R\widetilde H(x)s^{1-\gamma}\Big(\int_{1/b}^\infty   {V'\over V'^2+(U'-x')^2} dr\Big)dx.
\end{align*}

For $|x'|\geq 2\rho$,  by Lemma \ref{U'V'}, we obtain  $$ \int_{1/b}^{\infty}{V'\over V'^2+(U'-x')^2} dr\lesssim |x'|^{-1+1/\gamma}.$$

For $|x'|< 2\rho$, we show that a similar estimate holds. For this purpose, we only need to consider  $r$ large enough. By Lemma \ref{UV2}, we have $|U'-x'|\gtrsim U'$. Thus, $$
\int_{1/b}^{\infty}{V'\over V'^2+(U'-x')^2} dr\lesssim\int_{1/b}^{\infty}{V'\over V'^2+U'^2} dr\lesssim \int_{1/b}^{\infty}{V'\over U'^2} dr
\lesssim\int_{1/b}^{\infty} {1\over r^{\gamma+1}} dr\lesssim |x'|^{-1+1/\gamma}
$$
because $|x'|^{-1+1/\gamma}$ is bounded from below by $(2\rho)^{-1+1/\gamma}$.
Therefore,
$$\int_{1/b}^\infty s\widetilde  H(U+i V) dr\lesssim \int_\R  \widetilde H(x)s^{1-\gamma} |x'|^{-1+1/\gamma}dx =\int_\R  \widetilde H(x) |x|^{-1+1/\gamma}dx.$$ By \eqref{2.2}, the last integral is finite.

Then we deduce that $\int_{\S_1} H(\zeta) \pi^*(\ddc\|z\|^2)<\infty$. We can repeat the argument above for the other half sector $\S_2:=\{bu+av\leq v\}\cap \S$ and get $\int_{\S_2} H(\zeta) \pi^*(\ddc\|z\|^2)<\infty$, which finishes the proof of this proposition. Note that using the symmetric property of $\widetilde H$, we can also see that the integral on $\S_1$ is similar to the one on $\S_2$ and can be treated in the same way.
\endproof

\begin{proposition}\label{ddc}
The current $T$ is $\ddc$-closed.
\end{proposition}
\proof
Let $Q_s$ be the parallelogram in $\S$ limited  by $bu+av=s$ and $v=s$. We have 
$$T=\lim_{s\to\infty} T_s \quad \text{with}\quad T_s:=\pi_*(H[Q_s]).$$
For a smooth function $\phi$ with compact support in $\D^2$, we need to show that 
$\langle T_s, \ddc\phi\rangle\to 0$ as $s$ tends to infinity.

 Since $\phi$ is compactly supported in $\D^2$, there exists a positive constant $\lambda$ such that the support of $\pi^*(\phi)$ is contained in the sector $\S^\star:=\{v>\lambda, bu+av>\lambda\}$ which is contained in $\S$. Define $Q^\star:=Q_s\cap S^\star$.

By Stokes formula, we have
\begin{equation}\label{3.1}
\langle T_s, \ddc\phi \rangle =    \langle H\, d[Q_s^\star], \pi^*(d^c\phi)\rangle+ \langle dH\wedge [Q_s^\star], \pi^*(d^c\phi)\rangle.
\end{equation}

We show that the first term in \eqref{3.1} tends $0$. Observe that $d[Q_s^\star]=[E_s]+[E_s']$, where $E_s\subset \{v=s\}$ is the horizontal edge and $E_s'\subset \{bu+av=s\}$ is the vertical edge of $Q_s^\star$ inside $\S^\star$  with suitable orientations. We will only show that $\langle H[E_s], \pi^*(\dc\phi)\rangle$ tends to $0$. A similar property for $[E_s']$ can be obtained in the same way.

  Note that $\dc\phi$ is a combination with bounded coefficients of $dz_i$ and $d\overline z_i$. We also have
$$\pi^*(dz_1)=i\eta e^{i\eta\zeta} d\zeta,\quad \pi^*(d\overline z_1)=-i\overline\eta e^{-i\overline{\eta\zeta}}d\overline\zeta$$
and
$$\pi^*(dz_2)=ie^{i\zeta}d\zeta,\quad \pi^*(d\overline z_2)=-ie^{-i\overline \zeta}d\overline\zeta.$$
Observe that 
$$\abs{e^{i\eta\zeta}}=\abs{e^{-i\overline{\eta\zeta}}}=e^{-(bu+av)}\quad\text{and}\quad \abs{e^{i\zeta}}=\abs{e^{-i\overline\zeta}}=e^{-v}.$$ Since we are working with $\zeta=u+iv\in E_s$, we have that $e^{-v}\leq e^{-(bu+av)}$ and $d\zeta=d\overline\zeta=du$. 
So $\pi^*(\dc \phi)$ is equal to $e^{-(bu+av)}du$ times a bounded function.  We only have to check that $$\int_{E_s}H(\zeta)e^{-(bu+av)}du \to 0.$$
 Observe that $E_s=\{\lambda<bu+av\leq s,v=s\}$. Moreover, we have $bu+av=bsr$ on the half-line $\{v=s\}\cap \S$. Hence the considered integral is equal to \begingroup
 \allowdisplaybreaks
\begin{align}
&\int_{\lambda<bu+as\leq s}H(\zeta)e^{-(bu+as)} du=\int_{\lambda/(bs)}^{1/b}\widetilde H(U+iV)e^{-bsr}s dr \nonumber\\
&=\frac 1\pi\int_\R \widetilde H(x) \Big(\int_{\lambda/(bs)}^{1/b} s e^{-bsr} {V\over V^2+(U-x)^2} dr\Big) dx \nonumber\\
&=\frac 1\pi\int_\R \widetilde H(x) \Big(\int_{\lambda/(bs)}^{1/b} s^{1-\gamma}e^{-bsr}  {V'\over V'^2+(U'-x')^2} dr\Big) dx\nonumber\\
&\lesssim \int_\R \widetilde H(x) s^{1-\gamma}\Big(\int_{\lambda/(bs)}^{1/b} (bsr)^{-1}  {V'\over V'^2+(U'-x')^2} dr\Big) dx.  \label{3.2}
\end{align}\endgroup

We need to show that the expression in \eqref{3.2} tends to $0$. For this purpose, we will split the integral into two parts corresponding to $|x'|\geq 2\rho$ and $|x'|<2\rho$.

For $|x'|\geq 2\rho$, by Lemma \ref{U'V'} and using that $bsr>\lambda$ on $E_s$, we have 
\begingroup
\allowdisplaybreaks
\begin{align*}&\int_{|x'|\geq 2\rho} \widetilde H(x) s^{1-\gamma}  \Big(\int_{\lambda/(bs)}^{1/b} (bsr)^{-1} {V'\over V'^2+(U'-x')^2} dr\Big) dx \\
&\lesssim \int_{|x'|\geq 2\rho}  \widetilde H(x)s^{1-\gamma}|x'|^{-1+1/\gamma} dx
=  \int_{|x|\geq 2\rho s^\gamma} \widetilde H(x)|x|^{-1+1/\gamma}dx.
\end{align*}\endgroup
Since $\widetilde H(x)|x|^{-1+1/\gamma}$ is integrable on $\R$ (see \eqref{2.2}) and $2\rho s^\gamma\to \infty$, the last integral tends to $0$.

For $|x'|< 2\rho$,   when $s\to \infty$, we have $\lambda/(bs)\to 0$. Using Lemma \ref{UV1}, it is enough to estimate
\begingroup
\allowdisplaybreaks
\begin{align}
&\int_{|x'|< 2\rho} \widetilde H(x)s^{1-\gamma}\Big(\int_{\lambda/(bs)}^{1/b}   (bsr)^{-1} { r\over r^2+(\rho+x')^2} dr\Big) dx \nonumber\\
&\lesssim \int_{|x'|< 2\rho} \widetilde H(x)s^{1-\gamma}\Big(\int_{\lambda/(bs)}^{1/b} (bsr)^{-1} { r\over r^2} dr\Big) dx \nonumber \\
&\simeq \int_{|x|<2\rho s^\gamma}  \widetilde H(x)  s^{1-\gamma}  dx= \int_{|x|< 2\rho s^\gamma}  \widetilde H(x) |x|^{-1+1/\gamma} \Big({|x|\over s^\gamma}\Big)^{1-1/\gamma}dx.   \nonumber                   
\end{align}\endgroup
Note that the function inside the last integral converges pointwise to $0$. Therefore, by Lebesgue dominated convergence theorem, the last integral goes to $0$. Thus, the first term in \eqref{3.1} tends to $0$.

\vskip 5pt

Now we show that the second term  in \eqref{3.1} tends to $0$. Since $d=\partial+\dbar, \dc=\frac {i}{2\pi} (\partial-\dbar)$ and $H$ is harmonic, by Stokes formula, this term is equal to  
$$-\langle d H\wedge d^c[Q_s^\star], \pi^*(\phi)\rangle=\frac i{2\pi} \langle \partial H\wedge d[Q_s^\star], \pi^*(\phi)\rangle -\frac i{2\pi}
\langle \dbar H\wedge d[Q_s^\star], \pi^*(\phi)\rangle.$$

Recall that $d[Q_s^\star]=[E_s]+[E_s']$. We only consider the integral on the horizontal edge $E_s$ for simplicity. Fix a point $p$ on $E_s$, and denote by $d_{p,1}$ the distance from $p$ to the line $\{bu+av=0\}$ and $d_{p,2}$ the distance from $p$ to the line $\{v=0\}$. It is not hard to see that $d_{p,1}\simeq sr$ and $d_{p,2}=s$. Denote by $d_p:=\min(d_{p,1},d_{p,2})$.  Then $H$ is harmonic on the open disc $\D(p,d_p)$ of center $p$ and radius $d_p$, by Harnack's inequality, for $x\in\D(p,d_p)$ and $d_x:=\dist(x,p)$ we have
$${d_p-d_x\over d_p+d_x}H(p)\leq H(x)\leq {d_p+d_x\over d_p-d_x}H(p).$$
Using that $\lambda/(bs)<r\leq 1/b$ on $E_s$, we get that $d_p\simeq sr$ on $E_s$.
 Then by definition of derivative and taking $x\to p$, we deduce that at the point $p$,  $\partial H$ is equal to $(sr)^{-1}Hd\zeta$ times a bounded number and $\dbar H$ is equal to $(sr)^{-1}Hd\overline\zeta$ times a bounded number. Moreover, for $\zeta=u+iv \in E_s$, we have $d\zeta=d\overline \zeta =du$.
Therefore, it is enough to show that
 $$\int_{E_s}H(\zeta)(sr)^{-1}du\to 0.$$ Using the variables $U',V'$ and $x'$, we get
\begingroup
 \allowdisplaybreaks
 \begin{align*}
 &\int_{E_s}H(\zeta)(sr)^{-1}du=\int_{\lambda/(bs)}^{1/b}\widetilde H(U+iV)(sr)^{-1}sdr\\
 &=\frac 1\pi\int_\R \widetilde H(x) s^{1-\gamma}\Big(\int_{\lambda/(bs)}^{1/b} (sr)^{-1}  {V'\over V'^2+(U'-x')^2} dr\Big) dx.
 \end{align*}\endgroup
As for \eqref{3.2}, we see that the last integral tends to $0$ as $s$ tends to infinity. This ends the proof of the proposition.
\endproof

\begin{proposition}\label{geqdelta}
	For $0<\delta<1$, we have $\|T\|_{\delta\D^2}\geq \varepsilon(\delta)\delta^2$.
\end{proposition}

\proof
Recall that $\pi(\zeta):=(e^{i\eta\zeta}, e^{i\zeta})$. So $$\pi^{-1}(\delta\D^2)=\{u+iv:\,\,bu+av>-\log\delta,v>-\log\delta\}.$$
We define $t:=-\log\delta$.
Hence $$\|T\|_{\delta\D^2}\simeq \int_{bu+av>t,v>t}  H(\zeta) \pi^*(\ddc\|z\|^2).$$

To prove the proposition, it suffices to bound  the integral on $\S_1\cap\{bu+av>t,v>t\}$ from below. Using \eqref{3.0} and the variables $r$ and $s$ introduced in Section \ref{pre}, the considered integral is equal to a constant times
\begingroup
\allowdisplaybreaks
\begin{align*}
&\int_{v>t,bu+av\geq v}  H(\zeta)\big((a^2+b^2)e^{-2(bu+av)}+e^{-2v} \big)id\zeta\wedge d\overline\zeta\\
&\geq \int_{v>t,bu+av\geq v} H(\zeta)e^{-2v}id\zeta\wedge d\overline\zeta\\
&\simeq \int_t^\infty e^{-2s} \int_{1/b}^\infty \int_\R \widetilde H(x) s{V\over V^2+(U-x)^2}  dxdrds\\
&=\int_{t}^{\infty}e^{-2s}\int_\R \widetilde H(x)s^{1-\gamma}\Big(\int_{1/b}^{\infty}   {V'\over V'^2+(U'-x')^2} dr\Big)dxds\\
&\geq \int_{t}^{\infty}e^{-2s}\int_{x\geq s^\gamma} \widetilde H(x)s^{1-\gamma}\Big(\int_{1/b}^{\infty}   {V'\over V'^2+(U'-x')^2} dr\Big)dxds.
\end{align*}\endgroup
When $x\geq s^\gamma$, we have $x'\geq 1$.  Applying Lemma \ref{rgeq1} to the integral inside the parentheses, and using \eqref{2.2}, we obtain \begingroup
\allowdisplaybreaks
\begin{align*}
&\|T\|_{\delta\D^2}\gtrsim \int_{t}^{\infty}e^{-2s}\int_{x\geq s^\gamma}\widetilde H(x)s^{1-\gamma}x'^{-1+1/\gamma} dxds
=\int_{t}^{\infty}e^{-2s}\int_{x\geq s^\gamma}\widetilde H(x)x^{-1+1/\gamma} dxds\\
&\geq \int_{t}^{t+\log 2}e^{-2s}\int_{x\geq s^\gamma}\widetilde H(x)x^{-1+1/\gamma} dxds 
\geq \int_{t}^{t+\log 2}e^{-2s}\int_{x\geq (t+\log 2)^\gamma}\widetilde H(x)x^{-1+1/\gamma} dxds\\
&= \int_{t}^{t+\log 2}e^{-2s} A\varepsilon(e^{-t-\log 2}) ds
\geq \log 2\cdot e^{-2t-2\log 2}A\varepsilon(e^{-t-\log 2})
\simeq A\delta ^2\varepsilon(\frac 12 \delta)\geq \frac A2 \delta^2\varepsilon(\delta)
\end{align*}\endgroup
since $\varepsilon$ is concave and $\varepsilon(0)=0$. We get the desired estimate by taking $A$ large enough. The proof is complete.
\endproof

Summing up, Propositions \ref{well-define}, \ref{ddc} and \ref{geqdelta} show that the current $T$ defined in Section \ref{pre} is a positive  $\ddc$-closed current with the desired mass estimate. By definition, it is clear that $T$ has no mass on the separatrices at $0$ because $L_1\subset (\D^*)^2$. The proof of Theorem \ref{main} is finished. 

\begin{remark}\rm
	For every $\alpha\in \mathbb A$, we can construct a similar current $T_\alpha$ supported by $\overline L_\alpha$. By taking an average of those currents, we can have a current $T$ satisfying Theorem \ref{main} which is given by a smooth form on $(\D^*)^2$.
\end{remark}

\begin{remark}\rm
	Consider a general positive $\ddc$-closed current $T$ on $\D^2$ directed by $\mathscr F$. Assume that $T$ has no mass on the separatrices at $0$. It is known that $T$ can be written as an average of currents $T_\alpha$ supported by $\overline L_\alpha$ by mean of a positive measure on $\mathbb A$. As above, we can show that $T_\alpha$ is positive $\ddc$-closed for almost every $\alpha$.
	\end{remark}

\begin{remark}\rm
	The famous Siu's semicontinuity theorem \cite{siu} says that for every positive closed current $S$ on a complex manifold $X$, the upperlevel set $E_c:=\{\nu(S,\cdot)>c\}$ is an analytic subset of $X$ for every $c>0$. It is well-known that this property is not true for positive $\ddc$-closed current. For the current $T$ we constructed, we have $$\{\nu(T,\cdot)>c\}=\{\pi(z):\,\,H(z)>c\},$$ which is an open set of $L_1$ and is Zariski dense in $\D^2$.
\end{remark}

\medskip


\end{document}